\documentclass[a4paper,11pt]{amsart}
\usepackage{amssymb}
\usepackage{amsthm}
\usepackage{amsmath}
\usepackage{dsfont,graphpap}

\theoremstyle{plain}
\newtheorem{thm}{Theorem}[section]
\newtheorem{lem}[thm]{Lemma}
\newtheorem{prop}[thm]{Proposition}
\newtheorem{cor}[thm]{Corollary}

\theoremstyle{definition}
\newtheorem{defn}[thm]{Definition}

\theoremstyle{remark}

\newtheorem{rmq}[thm]{Remark}

\newcommand{\N}{\mathbb{N}}
\newcommand{\Z}{\mathbb{Z}}
\newcommand{\R}{\mathbb{R}}

\newcommand{\T}{\mathbb{T}}
\newcommand{\OO}{\mathcal{O}}
\newcommand{\eps}{\varepsilon}
\newcommand{\nl}{\left\|}
\newcommand{\nr}{\right\|}
\newcommand{\vphi}{\varphi}

\title{A symplectic non-squeezing theorem for BBM equation}
\author{David Roum\'egoux}
\address{University of Cergy-Pontoise, Department of Mathematics, CNRS, UMR 8088, F-95000 Cergy-Pontoise}
\email{david.roumegoux@u-cergy.fr}

\begin{document}

\begin{abstract}
We study the initial value problem for the BBM equation:

$$\left\{  
\begin{array}{l}
 u_t+u_x+uu_x-u_{txx}=0  \qquad x\in \T, t \in \R \\
u(0,x)=u_0(x)
\end{array}
\right.$$
We prove that the BBM equation is globaly well-posed on $H^s(\T)$ for $s\geq0$ and a symplectic non-squeezing theorem on $H^{1/2}(\T)$. That is to say the flow-map $u_0 \mapsto u(t)$ that associates to initial data $u_0 \in H^{1/2}(\T)$ the solution $u$ cannot send a ball into a symplectic cylinder of smaller width.

\end{abstract}

\maketitle

\section{Introduction}

In 1877 Joseph Boussinesq proposed a variety of models for describing the propagation of waves on shallow water surfaces, including what is now refered to as the Korteweg-de Vries (KdV) equation. A scaled KdV equation reads 
$$u_t+u_x+\eps(uu_x+u_{xxx})=0.$$

The Benjamin-Bona-Mahony (BBM) equation was introduced in \cite{BBM} as an alternative of the KdV equation. The main argument to derive the BBM equation is that, to
the first order in $\eps$, the scaled KdV equation is equivalent to
$$u_t+u_x+\eps(uu_x-u_{txx})=0.$$
Indeed, formally we have $u_t+u_x=O(\eps)$, hence $u_{xxx}=-u_{txx} + O(\eps)$.

In this article we shall consider the rescaled BBM equation:
 $$u_t+u_x+uu_x-u_{txx}=0.$$

In 2009, Jerry Bona and Nikolay Tzvetkov proved in \cite{BT} that BBM equation is globaly well-posed in $H^s(\R)$ if $s\geq 0$, and not even locally well-posed for negative values of $s$ (see also \cite{panthee}). The result extends to the periodic case (see section \ref{sec:BBMeq} below). Let us denote $\Phi_t$ the flow map of BBM equation on the circle $\T$. In this article we prove a symplectic non-squeeezing theorem for $\Phi_t$. That is, the flow map cannot squeeze a ball of radius $r$ of $H^{1/2}(\T)$ into a symplectic cylinder of radius $r'<r$. Precisely, let  $H^{1/2}_0(\T)=\left\lbrace u \in H^{1/2} / \int_{\T}  u = 0 \right\rbrace $ with the Hilbert basis
$$ \vphi_n^{+}(x)=\sqrt{\frac{n}{\pi(n^2+1)}} \cos(nx), \qquad \vphi_n^{-}(x)=\sqrt{\frac{n}{\pi(n^2+1)}} \sin(nx). $$
Set
$$B_r=\left\{ u \in H^{1/2}_0(\T) \; / \nl u \nr_{H^{1/2}} <r  \right\},$$
$$\mathcal{C}_{r,n_0}=\left\{ u=\sum p_n \vphi_n^+ + q_n \vphi_n^- \in H^{1/2}_0(\T) \; / p_{n_0}^2 + q_{n_0}^2 <r^2 \right\}.$$
The goal of this paper is to prove
\begin{thm}
\label{thm:BBMnonsqueezing}
 If $\Phi_t(B_r) \subset \mathcal{C}_{R,n_0}$ then $r\leq R$. 
\end{thm}

S. Kuksin initiated the investigation of non-squeezing results for infinite dimentional Hamiltonian systems (see \cite{kuksin}). In particular he proved that nonlinear wave equation has the non-squeezing property for some nonlinearities. This result were extended to certain stronger nonlinearities by Bourgain \cite{bourgain}, and he also proved with a different method that the cubic NLS equation on the circle $\T$ has the non-squeezing property. Using similar ideas Colliander, Keel, Staffilani, Takaoka and Tao obtained the same result for KdV equation on $\T$ (see \cite{CKSTT}).

In this article we will use the original theorem of Kuksin. In section \ref{sec:DI}, we present the construction of a capacity on Hilbert spaces introduced by Kuksin in \cite{kuksin}. This capacity is invariant with respect to the flow of some hamiltonian PDEs provided it has the form ``linear evolution + compact''. As a corollary of this result we get a non-squeezing theorem for these PDEs. Then we apply this theorem to the BBM equation in section \ref{sec:BBMeq}. We prove the global wellposedness of BBM equation on $H^s(\T)$ for $s\geq 0$, and some estimates on the solutions.

\section{Symplectic capacities in Hilbert spaces and non-squeezing theorem}
\label{sec:DI}

\subsection{The frame work and an abstract non-squeezing theorem}
\label{subsec:assumptions}

Let $(Z, \left\langle \cdot,\cdot \right\rangle )$ be a real Hilbert space with $\{ \varphi_j^{\pm} / j\geq 1 \}$ a Hilbert basis. For $n\in \N$ we denote $Z^n=\mbox{Span}(\{ \varphi_j^{\pm} / 1\leq j \leq n \})$, and $\Pi^n:Z\to Z^n$ the corresponding projector. We also denote $Z_n$ the space such that $Z=Z^n\oplus~Z_n$. Then, every $z\in Z$ admits the unique decomposition $z=z^n+z_n$ with $z_n \in Z_n$ and $z^n \in Z^n$.

We define $J:Z \to Z$ the skewsymmetric linear operator by 
$$J\varphi_j^\pm=\mp \varphi_j^\mp$$
and we supply $Z$ with a symplectic structure with the $2$-form $\omega$ defined by $\omega(\xi,\eta)= \left\langle J\xi,\eta \right\rangle $.

We take a self-adjoint operator $A$, such that 
\begin{equation} \label{eqn:ADZ} \forall j \in \Z,\; A \varphi_j^\pm=\lambda_j \varphi_j^\pm . \end{equation}
Define the Hamiltonian
$$f(z)=\frac{1}{2} \left\langle Az,z \right\rangle +h(z)$$
where $h$ is a smooth function defined on $Z \times \R$. The corresponding Hamiltonian equation has the form
\begin{equation} \label{eqn:EDP} \left\lbrace \begin{array}{l}  \dot{z}=JAz+J\nabla h(z) \\ z(0,\cdot)=z_0 \in Z  \end{array} \right. \end{equation}

If $Z_-$ is a Hilbert space, we denote 
$$Z<Z_-$$
if $Z$ is compactly embedded in $Z_-$ and $\{\varphi_j^\pm\}$ is an orthogonal basis of $Z_-$ (not an orthonormal one!). Clearly $Z$ is dense in $Z_-$. We identify $Z$ and its dual $Z^*$. Then $(Z_-)^*$ can be identified with a subspace $Z_+$ of $Z$ and we have 
$$Z_+<Z<Z_- .$$  
Denote $\|\cdot\|_-$ (resp. $\|\cdot\|_+$) the norm of $Z_-$ (resp. $Z_+$).

We also denote $B_R(Z)$ the ball centered at the origin of radius $R$. 

We impose the following assumptions:

\begin{description}
 \item [(H1)]  The equation (\ref{eqn:EDP}) defines a $C^1$-smooth global flow map $\Phi$ on $Z$. That is, for all $z_0 \in Z$ the equation (\ref{eqn:EDP}) has a unique solution $z(t)=\Phi_t(z_0)$ for $t\geq0$, and the flow map $\Phi_t:z_0 \mapsto z(t)$ is $C^1$-smooth.
 \item [(H2)]  The flow map $\Phi$ is uniformely bounded. That is for each $R>0$ and $T>0$, there exists $R'=R'_{R,T}$ such that 
$$\Phi_t(B_R(Z)) \subset B_{R'}(Z), \quad \mbox{ for }  |t| \leq T .$$
 \item [(H3)]  Writing the flow map $\Phi_t=e^{tJA}(I+\widetilde{\Phi}_t)$, we also impose the following \emph{compactness assumption} : fix $R>0$ and $T>0$, there exists $C_{R,T}$ such that
$$\forall u_0,u_0' \in B_R(Z), \; \nl \widetilde{\Phi}_T(u_0)-\widetilde{\Phi}_T(u_0') \nr_{Z_+} \leq C_{R,T} \nl u_0-u_0' \nr_{Z} .$$
\end{description}
Under these assumptions, it is well known that the flow maps $\Phi_t$ preserve the symplectic form.

The aim of this section is to show the following non-squeezing theorem 
\begin{thm}
\label{thm:nonsqueezing}
Assume $\Phi_T$ is the flow map of an equation of the form (\ref{eqn:EDP}) and satisfies the previous assumptions. If $\Phi_T$ sends a ball 
$$B_r=\{z \in Z / \|z-\bar{z}\|<r\},\quad \bar{z} \mbox{ fixed}$$
 into a cylinder 
\begin{multline*}
\mathcal{C}_{R,j_0}= \left\{ \left. z= \sum p_j \varphi_j^++q_j \varphi_j^-  \right/ (p_{j_0}-\bar{p}_{j_0})^2+(q_{j_0}-\bar{q}_{j_0})^2 < R^2 \right\} \\ j_0,\bar{p}_{j_0},\bar{q}_{j_0} \mbox{ fixed}
\end{multline*}
then $r\leq R$.
\end{thm}
In fact, this theorem is a simple version of the conservation of a symplectic capacity on $Z$ by the flow map $\Phi_T$ (see subsection \ref{subsec:caphilbert} below) 

\begin{rmq}
 This theoreme implies the following fact. Fiw $\eps>0$, a time $T>0$, a Fourier mode $n_0$ and $r>0$ (no smallness conditions are imposed on $r$ or $T$), then there exists $u_0 \in H^{1/2}(\T)$ such that 
$$\nl u_0 \nr_{H^{1/2}} < r$$
and 
$$|\widehat{u(T)}(n_0)| > \frac{r-\eps}{(n_0^2 + 1)^{1/4}}$$
where $u$ solves (\ref{eqn:EDP}).

The non-squeezing theorem remains true if we don't suppose that the flow map is global in (H1), but the conclusion would be :\\
either 
$$|\widehat{u(T)}(n_0)| > \frac{r-\eps}{(n_0^2 + 1)^{1/4}}$$
or 
$$\sup_{0\leq t \leq T} \nl u(t) \nr_{H^{1/2}}=+\infty.$$
So we impose the global wellposedness in time for (\ref{eqn:EDP}) in order to rule out the second case.
\end{rmq}

\subsection{An approximation lemma}

In order to define a capacity, we will need to approximate the flow by finite-dimentional maps. We shall use the following lemma

\begin{lem}
\label{lem:decompoPhi}
Let $\Phi$ the flow at time $T$ of an equation (\ref{eqn:EDP}) satisfying the previous assumptions. For each $\varepsilon>0$ and $R>0$, there exists $N \in \N$ such that for $u\in B_R$ :
\begin{equation} \label{eqn:decompoPhi}
\Phi(u)=e^{tJA}(I+\widetilde{\Phi}_\varepsilon)(I+\widetilde{\Phi}_N)(u)
\end{equation}
where $(I+\widetilde{\Phi}_\varepsilon)$ and $(I+\widetilde{\Phi}_N)$ are symplectic diffeomorphisms satisfying
\begin{equation} \label{eqn:decompoPhi2}
\|\widetilde{\Phi}_\varepsilon(u)\|\leq \varepsilon \qquad \mbox{for } u \in (I+\widetilde{\Phi}_N)(B_R) 
\end{equation}
\begin{equation} \label{eqn:decompoPhi3}
\left(I+\widetilde{\Phi}_N\right)(u^N+u_N)=\left(I+\widetilde{\Phi}_N\right)(u^N)+u_N \quad \mbox{ for } u^N \in Z^N, u_N \in Z_N .
\end{equation}
\end{lem}

\begin{proof}
Recall that $\Phi=e^{TJA}(I+\widetilde{\Phi})$. First, we observe that for $|t| \leq T$, any $R>0$ and $u,v \in B_R(Z)$ we have
\begin{equation}
 \label{eqn:estimPhi}
\nl \widetilde{\Phi}(u) - \Pi^N \widetilde{\Phi}(u) \nr_Z \leq \eps_1(N) \underset{N \to + \infty}{\longrightarrow} 0 .
\end{equation}
Indeed, as $K =\bigcup_{|t| \leq T} \widetilde{\Phi}(B_r(Z))$ is precompact in $Z$ (by (H3)), then (\ref{eqn:estimPhi}) results from the following statement 
$$\sup_{u \in K} \nl u - \Pi^N u \nr \underset{N \to + \infty}{\longrightarrow} 0 .$$
Suppose that the convergence does not hold, then we can find a sequence $(u_n)$ in $K$ such that $\nl (I-\Pi^n)u_n \nr \geq \eps >0$. As $K$ is precompact there exists a subsequence $(u_{n_j})$ such that $u_{n_j} \rightarrow u$. For $n_j$ sufficiently large we have
$$\nl (I-\Pi^{n_j})(u) \nr \leq \eps/2, \quad \nl u_{n_j}-u\nr \leq \eps/2 .$$
Hence $\nl (I-\Pi^{n_j})(u_{n_j}) \nr \leq \eps $ and we get a contradiction.

Now we set $h_N=h \circ \Pi^N$. Then $\nabla h_N= \Pi^N \nabla h \Pi^N$. We define $\Phi^N$ the time $T$ flow of the equation 
\begin{equation}
 \label{eqn:EDPDF}
\dot{v}=J(Av+\nabla h_N (v))
\end{equation}
or, equivalently, $v=v^N+v_N\in Z^N + Z_N$ and
$$\left\lbrace \begin{array}{l}  \dot{v}^N=J(Av^N+\Pi^N \nabla h(v^N)) \\ \dot{v}_N=JAv_N \end{array} \right.$$
We write $\Phi^N=e^{TJA}(I+ \widetilde{\Phi}_N) .$

Since $\widetilde{\Phi}_N=0$ outside $Z^N$, $\widetilde{\Phi}_N$ has the desired form (\ref{eqn:decompoPhi3}). Define 
$$\widetilde{\Phi}_{\eps}=\left(\widetilde{\Phi}-\widetilde{\Phi}_N\right) \left(I+\widetilde{\Phi}_N\right)^{-1} ,$$
so we have
$$
 e^{TJA}\left(I+\widetilde{\Phi}_{\varepsilon} \right) \left(I+\widetilde{\Phi}_N\right) =e^{TJA}\left(I+\widetilde{\Phi}\right) = \Phi .
$$

Next we estimate the difference $\widetilde{\Phi}-\widetilde{\Phi}_N$. For $u \in B_R(Z)$ we have
\begin{multline*}
\nl    \widetilde{\Phi}(u) -\widetilde{\Phi}_N(u) \nr_Z  \leq \nl\widetilde{\Phi}(u)-\Pi^N \widetilde{\Phi}(u) \nr_Z + \nl \Pi^N \widetilde{\Phi}(u)-\Pi^N \widetilde{\Phi}(\Pi^N u) \nr_Z \\+ \nl \Pi^N \widetilde{\Phi}(\Pi^N u) - \widetilde{\Phi}_N(u) \nr_Z .
\end{multline*}
Hence by (\ref{eqn:estimPhi}) and assumption (H3), for $u \in B_R(Z)$ we have 
$$\nl \widetilde{\Phi}(u)-\widetilde{\Phi}_N(u) \nr_Z \leq C \eps(N) \underset{N \to + \infty}{\longrightarrow} 0 ,$$
so for $u \in \left(I+\widetilde{\Phi}_N\right)(B_R(Z))$
$$\nl \widetilde{\Phi}_{\eps}(u) \nr_Z \leq \eps(N) \underset{N \to + \infty}{\longrightarrow} 0 .$$
\end{proof}

\subsection{Symplectic capacities and non-squeezing theorem}
\subsubsection{Capacities in finite-dimentional space}

Consider $\R^{2n}$ supplied with the standard symplectic structure, that is $\omega(x,y)=\left\langle Jx,y \right\rangle$ where 
$$J=\left( \begin{array}{cc} 0 & -I \\ I & 0 \end{array} \right).$$
For $f:\R^{2n} \rightarrow \R$ a smooth function we define the hamiltonian vectorfield 
$$X_f=J \nabla f.$$

\begin{defn} \label{def:fadmDF}
Let $\OO$ an open set of $\R^{2n}$, $f \in C^\infty(\OO)$ and $m>0$. The function $f$ is called \emph{m-admissible} if
\begin{itemize}
\item $0\leq f(x) \leq m$ for $x \in \OO$, and $f$ vanishes on a nonempty open set of $\OO$, and $f|_{\partial \OO}=m$.
\item The set $\{z/f(z)<m\}$ is bounded and the distance from this set to $\partial \OO$ is $d(f)>0$.
\end{itemize}
\end{defn}

Following \cite{HZ} we define the capacity $c_{2n}(\OO)$ of an open set $\OO$ of $\R^{2n}$ as
\begin{align*}
& c_{2n}(\OO)  =\inf \left\{ m_* / \mbox{for each } m>m_* \mbox{ and each } m\mbox{-admissible function } f \mbox{ in } \OO \right. \\
& \left. \mbox{the vectorfield } X_f \mbox{ has a non constant periodic solution of period } \leq 1 \right\}.
\end{align*}

\begin{thm}
 \label{thm:c2n}
$c_{2n}$ is a symplectic capacity, that is
\begin{itemize}
 \item if $\OO_1 \subset \OO_2$ then $c_{2n}(\OO_1) \leq c_{2n}(\OO_2)$ \\
and if $\vphi:\OO \rightarrow \R^{2n}$ is a symplectic diffeomorphism then  $c_{2n}(\OO)=c_{2n}(\vphi(\OO))$.
 \item $c_{2n}(\lambda \OO)=\lambda^2 c_{2n}(\OO)$.
 \item $c_{2n}(B_1) =c_{2n}(\mathcal{C}_{r,1})=\pi$ where 
$$B_r=\left\{(p,q)/\sum(p_j^2+q_j^2) <r^2 \right\}, \mbox{ and } \mathcal{C}_{r,1}=\left\{(p,q)/(p_1^2+q_1^2) <r^2 \right\}.$$
\end{itemize}
\end{thm}

See \cite{HZ} for a proof. An immediate consequence of this theorem is the non-squeezing theorem of M. Gromov \cite{gromov}.

\begin{thm}
 \label{thm:gromov}
The ball $B_r$ can be symplecticaly embedded into the cylinder $\mathcal{C}_{R,1}$ if and only if $r\leq R$.
\end{thm}

\subsubsection{Construction of a capacity on Hilbert spaces}
\label{subsec:caphilbert}
In this section we define a symplectic capacity on Hilbert spaces which is invariant with respect to the flow of the equation (\ref{eqn:EDP}). We will follow the construction of S. Kuksin (see \cite{kuksin}).

For $\OO$ an open set of $Z$ we denote $\OO^n=\OO \cap Z^n$ and observe that $\partial \OO^n \subset \partial \OO \cap Z^n$.

\begin{defn} \label{def:fadm}
Let $f \in C^\infty(\OO)$ and $m>0$. The function $f$ is called \emph{m-admissible} if
\begin{itemize}
\item $0\leq f(x) \leq m$ for $x \in \OO$, and $f$ vanishes on a nonempty open set of $\OO$, and $f|_{\partial \OO}=m$.
\item The set $\{z/f(z)<m\}$ is bounded and the distance from this set to $\partial \OO$ is $d(f)>0$.
\end{itemize}
\end{defn}

\begin{rmq} \label{rmq:d(f)}
If $f$ is $m$-admissible, denoting $\mbox{supp}(f)=\{z/0<f(z)<m\}$ we have
\begin{equation*} \begin{split}
\mbox{dist}(f^{-1}(0),\partial \OO)\geq d(f) , \\
\mbox{dist}(\mbox{supp}(f),\partial \OO) \geq d(f).
\end{split}
\end{equation*}
\end{rmq}

Denote $f_n=f|_{\OO^n}$ and consider $X_{f_n}$ the corresponding hamiltonian vectorfield on $\OO^n$. 
\begin{defn}
A $T$-periodic trajectory of $X_{f_n}$ is called \emph{fast} if it is not a stationnary point and $T\leq 1$.

A $m$-admissible function $f$ is called \emph{fast} if there exists $n_0$ (depending on $f$) such that for all $n\geq n_0$ the vectorfield $X_{f_n}$ has a fast solution.
\end{defn}

\begin{lem} \label{lem:solrapide}
Each periodic trajectory of $X_{f_n}$ is contained in $\mbox{supp}(f) \cap Z^n$.
\end{lem}

\begin{proof}
Pick $z\in \OO^n \backslash \mbox{supp}(f)$, $f_n$ takes either its minimal or maximal value in $z$, hence $X_{f_n}(z)=0$. Therefore $z$ is a stationnary point and a fast trajectory cannot pass through it. 
\end{proof}

We are now in position to define a capacity $c$.

\begin{defn}
For an open set $\OO$ of $Z$ its capacity equals to
$$c(\OO)=\inf \left\{m_* / \mbox{each } m \mbox{-admissible function with } m>m_* \mbox{ is fast} \right\}.$$
\end{defn}

\begin{prop} \label{prop:A1}
Assume that $\OO_1$, $\OO_2$ and $\OO$ are open sets of $Z$ and $\lambda\neq 0$
\begin{enumerate}
\item if $\OO_1 \subset \OO_2$ then $c(\OO_1) \leq c(\OO_2)$ ;
\item $c(\lambda \OO)=\lambda^2 c(\OO)$.
\end{enumerate} 
\end{prop}

\begin{proof}
(1) Assume $m<c(\OO_1)$, by definition of $c$ there exists a $m$-admissible function $f$ of $\OO_1$ which is not fast. Hence, there exists a sequence $(n_j) \rightarrow +\infty$ such that for every $j \in \N$, $X_{f_{n_j}}$ has no fast periodic trajectory. Define $\widetilde{f}$ on $\OO_2$ by 
$$\widetilde{f}(x)=\left\{
\begin{array}{ll}
f(x) & \mbox{if }x\in \OO_1 \\
m & \mbox{otherwise}
\end{array} \right.
$$
The function $\widetilde{f}$ is clearly $m$-admissible on $\OO_2$. \\
By lemma \ref{lem:solrapide}, for each $j \in \N$, each fast solution $x(t)$ of $X_{\widetilde{f}_{n_j}}$ lies in supp$\widetilde{f} \cap Z^{n_j}=$supp$f \cap Z^{n_j}$. Hence $x(t)$ is a fast trajectory of $X_{f_{n_j}}$ ($X_{\widetilde{f}_{n_j}}$ and $X_{f_{n_j}}$ are the same vectorfields on supp$(f)$ by definition of supp$(f)$).\\
Therefore, for each $j \in \N$ the vectorfield $X_{\widetilde{f}_{n_j}}$ of $\OO_2$ has no fast trajectory. Hence $\widetilde{f}$ is $m$-admissible but is not fast. Thus $c(\OO_2) \geq m$, and the first assertion follows.

(2) Define $f^\lambda=\lambda^2f(\lambda^{-1} \cdot)$  on $\lambda \OO$. Clearly $f$ is $m$-admissible on $\OO$ if and only if $f^\lambda$ is $\lambda^2 m$-admissible on $\lambda \OO$. Moreover $z(t) \in \OO^n$ is a $T$-periodic trajectory of $X_{f_n}$ if and only if $\lambda z(t) \in \lambda \OO^n$ is a $T$-periodic trajectory of $X_{f_n^\lambda}$. Therefore $c(\lambda \OO)=\lambda^2 c(\OO)$.
\end{proof}

\begin{lem} \label{lem:invarcdf}
If $F:Z \to Z$ has the form
$$F(z^n+z_n)=F^n(z^n)+z_n \qquad z=z^n+z_n \in Z=Z^n \oplus Z_n$$
with $F^n$ a symplectic diffeomorphism of $Z^n$, then $c(\OO)=c(F(\OO))$, for each open set $\OO$ of $Z$.
\end{lem}

\begin{proof}
We observe that if $f$ is $m$-admissible in $F(\OO)$ and $f$ is fast then $f \circ F$ is $m$-admissible in $\OO$ and $f \circ F$ is fast. Indeed $F^*:f\mapsto f\circ F$ clearly sends $m$-admissible functions in $F(\OO)$ to similar ones in $\OO$, and for $p \geq n$ it tranforms $X_{{(f \circ F)}^p}$ into $X_{f^p}$. Hence admissible and fast functions are preserved by $F$ and its inverse ($F$ is the identity outside of $Z^n$ which is a finite-dimentional space), and the result follows.
\end{proof}

\begin{prop} \label{prop:invartranslation}
For each open set $\OO$ of $Z$ and $\xi$ in $Z$, we have
$$c(\OO)=c(\OO+\xi) .$$
\end{prop}

\begin{proof}
Denote $\OO_\xi=\OO+\xi$. It is sufficient to prove that $c(\OO)\leq c(\OO+\xi)$ (change $\xi$ into $-\xi$).\\
Denote $\xi=\xi^{n_0}+\xi_{n_0} \in Z^{n_0} + Z_{n_0}$ ($n_0$ will be fixed later) and $\OO_1=\OO+\xi^{n_0}$.By lemma \ref{lem:invarcdf} $c(\OO_1)=c(\OO)$. We also remark that $\OO_\xi=\OO_1+\xi_{n_0}$.

Take any $m$-admissible function $f$ on $\OO_\xi$ with $m>c(\OO)$. We wish to check that $f$ is fast. \\
Since $\partial \OO_\xi \subset \partial \OO_1 + \xi_{n_0}$ and $\|\xi_n\| \underset{n\to+\infty}{\longrightarrow} 0$, we have
$$\mbox{dist}(\partial \OO_1,\partial \OO_\xi)  \leq \mbox{dist}(\partial \OO_1,\partial \OO_1 + \xi_{n_0})  \leq \|\xi_{n_0}\| \underset{n_0\to+\infty}{\longrightarrow} 0 .$$
Pick $n_0$ such that 
\begin{equation} \label{eqn:do1doxi} \mbox{dist}(\partial \OO_1,\partial \OO_\xi) \leq \|\xi_{n_0}\| < \frac{1}{2}d(f) . \end{equation} 
We extend $f$ outside $O_\xi$ with $f(z)=m$ if $z \notin \OO_\xi$ and we denote $\widetilde{f}$ its restriction to $\OO_1$.

$f$ equals $m$ on a $d(f)$-neighbourhood of $\partial \OO_\xi$. By (\ref{eqn:do1doxi}), we deduce that $\widetilde{f}$ equals $m$ on a $\frac{1}{2}d(f)$-neighbourhood of $\partial \OO_1$.

By remark \ref{rmq:d(f)} we have  $\mbox{dist}(f^{-1}(0),\partial \OO_\xi)\geq d(f)$. Hence, by (\ref{eqn:do1doxi}),we have $\mbox{dist}(f^{-1}(0),\partial \OO_1)\geq \frac{1}{2}d(f)$, and in particular $\widetilde{f}$ vanishes on a nonempty open set of $\OO_1 \cap \OO_\xi \subset \OO_1$. Therefore $\widetilde{f}$ is $m$-admissible. 

Since $c(\OO_1)=c(\OO)< m$, it follows that $X_{\widetilde{f}_n}$ has a fast trajectory in $\OO_1^n$ if $n\geq n_0$ is sufficiently large. By lemma \ref{lem:solrapide} this trajectory lies in $\mbox{supp}\widetilde{f}=\mbox{supp}f \subset \OO_1 \cap \OO$. Hence this trajectory is a fast solution of $X_{f_n}$, and the function $f$ is fast.
\end{proof}

If $\boldsymbol{r}=(r_j)_{j \in \N^*}$ is a sequence of $\R_+^* \cup\{+\infty\}$ with $\displaystyle 0< r=\inf_{j\in\N^*}r_j<+\infty$, we define
$$D(\boldsymbol{r})= \left\{ \left. z= \sum_{j=1}^{+\infty} p_j \varphi_j^++q_j \varphi_j^-  \right/  \forall j \in \N,\; p_j^2+q_j^2 < r_j^2 \right\} ,$$
$$E(\boldsymbol{r})= \left\{ \left. z= \sum_{j=1}^{+\infty} p_j \varphi_j^++q_j \varphi_j^-  \right/  \sum_{j=1}^{+\infty} \frac{p_j^2+q_j^2}{r_j^2} <1 \right\} .$$

Remark that if $\boldsymbol{r}=(r,+\infty,\ldots,+\infty)$, $D(\boldsymbol{r})$ is a symplectic cylinder $\mathcal{C}_{r,1}$.

\begin{thm} \label{thm:capellipsoide}
We have $c(E(\boldsymbol{r}))=c(D(\boldsymbol{r}))=\pi r^2$
\end{thm}

\begin{proof}
We have to check the following inequalities 
\begin{enumerate}
	\item $c(E(\boldsymbol{r})) \geq \pi r^2$
	\item $c(D(\boldsymbol{r})) \leq \pi r^2$
\end{enumerate}
then we will conclude by proposition \ref{prop:A1}.

(1) It is sufficient to prove that $c(B_1) \geq \pi$ (then the result follows by proposition \ref{prop:A1}).

Define $m=\pi-\varepsilon$. Choose $f:[0,1] \rightarrow \R_+$ satisfying : \\
$\left\{
\begin{array}{l}
0 \leq f '(t) < \pi \mbox{  for } t \in [0,1]   \\
f(t)=0 \mbox{  for } t \mbox{ near } 0   \\
f(t)=m \mbox{  for } t \mbox{  near } 1 
\end{array}
\right.$

Then, define $H(x)=f(\|x\|^2)$ for $x$ in $B(1)$. $H$ is $m$-admissible. We want to prove that $H$ is not fast. Consider
$$H_n(x)=f \left(\sum_{j=1}^n (p_j^2+q_j^2)\right), \quad \mbox{where } x=\sum_{j} (p_j \varphi_j^++q_j \varphi_j^-) .$$

Using the variables $I_j=\frac12(p_j^2+q_j^2)$ and $\theta_j=\arctan \left( \frac{p_j}{q_j} \right)$ we observe that non-constant periodic solutions corresponding to this hamiltonian has a period $T>1$. Hence $X_{H_n}$ has no fast trajectory and $H$ is not fast.

(2) Denote $\OO=D(\boldsymbol{r})$. Pick $m>\pi r^2$ and $f$ a $m$-admissible function in $\OO$. Since $f^{-1}(0)$ is not empty, there exists $n$ such that $f^{-1}(0) \cap Z^n \neq \emptyset$. Denote $f_n=f|_{\OO^n}$. Since $\partial \OO^n \subset \partial \OO$, we deduce that $f_n$ equals $m$ on a neighbourhood of $\partial \OO^n$. Hence $f_n$ is $m$-admissible.

Since $\displaystyle c_{2n}(\OO^n)=\pi \min_{1\leq j \leq n} r_j^2$ , we have
$$c_{2n}(\OO^n) \underset{n \to +\infty}{\longrightarrow} \pi \inf_{j \geq 1} r_j^2=\pi r^2 < m .$$
Hence, for $n$ sufficiently large $c_{2n}(\OO^n)< m$. Therefore $X_{f_n}$ has a fast periodic trajectory and the function $f$ is fast.
\end{proof}

\begin{cor} \label{cor:capboule}
We have $c(B_r)=c(\mathcal{C}_{r,1})= \pi r^2$,

and for each bounded open set $\OO$ of $Z$ we have $0<c(\OO)<+\infty $.
\end{cor}

The essential property of the capacity $c$ is its invariance with respect to the flow maps of PDEs satisfying assumptions (H1), (H2) and (H3). In fact the non-squeezing theorem \ref{thm:nonsqueezing} is a consequence of the following result.

\begin{thm} \label{thm:invarc}
Let $\Phi_T$ the flow of an equation (\ref{eqn:EDP}) satisfying the assumptions (H1), (H2) and (H3). For any open set $\OO$ of $Z$ we have 
$$c(\Phi_T(\OO))=c(\OO) .$$
\end{thm}

\begin{proof}
Let us denote $\Phi=\Phi_T$ and $\mathcal{Q}=\Phi(\OO)$. One easily checks that $\Phi^{-1}$ satisfies (H1), (H2) and (H3), therefore it is sufficient to prove that $c(\mathcal{Q})\leq c(\OO)$.

Take any $m>c(\OO)$ and any $f$ $m$-admissible in $\mathcal{Q}$. We want to prove that $f$ is fast.

Since $f$ is $m$-admissible there exists $R>0$ such that $\mbox{supp}f \subset B_R$. Define $R_1=R+d(f)$, $\mathcal{Q}' =\mathcal{Q} \cap B_{R' }$ and $\OO' =\Phi^{-1}(\mathcal{Q}' )$. By assumption $\OO' $ is bounded, hence there exists $R' $ such that $\OO' \subset B_{R' }$. Moreover we clearly have $\OO' \subset \OO$, thus by proposition \ref{prop:A1} 
\begin{equation}
\label{eqn:OO'}
c(\OO' ) \leq c(\OO) .
\end{equation}

We apply lemma \ref{lem:decompoPhi} with $N$ so large that $\varepsilon<\frac{1}{2}d(f)$, and we use the notations of the lemma \ref{lem:decompoPhi} :  $\Phi=e^{TJA}(I+\widetilde{\Phi}_\varepsilon)(I+\widetilde{\Phi}_N)$. We denote $\OO_1$ and $\OO_2$ the intermediate domains which arrise from the decomposition
$$\OO' \xrightarrow{I+\widetilde{\Phi}_N} \OO_1 \xrightarrow{I+\widetilde{\Phi}_\varepsilon} \OO_2 \xrightarrow{e^{TJA}} \mathcal{Q}' .$$
We also denote
$$f_2={\left. \left( f \circ e^{TJA} \right) \right|}_{\OO_2} .$$

Observe that $f_2$ is $m$-admissible on $\OO_2$. Indeed $f$ is $m$-admissible on $\mathcal{Q}$ and also on $\mathcal{Q}' $ (by definition of $\mathcal{Q}' $). Since $e^{tJA}$ is an isometry, $f_2$ is $m$-admissible. 

Then, we extend $f_2$ as $m$ outside $\OO_2$, and we denote $\widetilde{f}$ its restriction to $\OO_1$. By (\ref{eqn:decompoPhi2}) the $\varepsilon$-neighbourhood of $\partial \OO_1$ is contained in the $2\varepsilon$-neighbourhood of $\partial \OO_2$. Since $\varepsilon<\frac12 d(f)$, we deduce that $\widetilde{f}$ equals $m$ on a neighbourhood of $\partial \OO_1$. Moreover $\widetilde{f}^{-1}(0)=f_2^{-1}(0) \subset \OO_1 \cap \OO_2$. Indeed by remark \ref{rmq:d(f)} 
$$\mbox{dist}(f_2^{-1}(0),\partial \OO_2) \geq d(f)$$
$$\mbox{and   } \mbox{dist}(\partial \OO_1,\partial \OO_2)\leq \frac12 d(f) .$$

Hence $\widetilde{f}$ is $m$-admissible on $\OO_1$.

Using lemma \ref{lem:invarcdf} and (\ref{eqn:OO'}), we deduce that
$$c(\OO_1)=c\left((I+\widetilde{\Phi}_N)(\OO' ) \right)=c(\OO' ) \leq c(\OO) < m .$$
Hence $\widetilde{f}$ is $m$-admissible on $\OO_1$ and $c(\OO_1)< m$, thus $\widetilde{f}$ is fast. So for $n$ sufficiently large, the vectorfield $X_{\widetilde{f}_n}$ (where $\widetilde{f}_n=\widetilde{f}|_{\OO_1^n}$) has a fast solution. By lemma \ref{lem:solrapide} this solution lies in supp$\widetilde{f}$ and by remark \ref{rmq:d(f)} supp$\widetilde{f}=$supp$f_2$, so this solution is also a fast solution of $X_{f_2^n}$ (where $f_2^n={f_2}|_{\OO_2^n}$). Hence $f_2$ is fast too. Finally $f$ is also fast ($f_2={\left. \left( f \circ e^{TJA} \right) \right|}_{\OO_2}$).
\end{proof}

\section{Application to the BBM equation}
\label{sec:BBMeq}
In this section we prove that the BBM equation
\begin{equation}
\left\{ \begin{array}{l}
u_t+u_x+uu_x-u_{xxt}=0, \quad x\in \T \\
u(0,x)=u_0(x)
\end{array}
\right.
\label{eq:bbm}
\end{equation}
is globally well-posed in $H^s(\T)$ for $s\geq 0$ (we will follow the proof given in \cite{BT} for $x\in \R$) and has the non-squeezing property (theorem \ref{thm:BBMnonsqueezing}).

\subsection{Bilinear estimates}

We start by two helpful inequalities. \\ Let $\vphi(k)=\frac{k}{1+k^2}$ and $\vphi(D)$ the Fourier multiplier operator defined by $\widehat{\vphi(D)u}(k)=\vphi(k) \widehat{u}(k)$.

 \begin{lem}
\label{lem:bilin1}
Let $u \in H^r(\T)$ and $v\in H^{r'}(\T)$ with $0 \leq r \leq s$, $0 \leq r' \leq s$ and $0\leq 2s-r-r' <1/4$. Then 
$$\nl \vphi(D)(uv) \nr_{H^s} \leq C_{r,r',s} \nl u \nr_{H^r} \nl v \nr_{H^{r'}} $$
\end{lem} 

\begin{proof}
We want to prove 
$$\nl \left\langle k\right\rangle^s \frac{k}{1+k^2} \widehat{uv}(k) \nr_{\ell^2_k} \leq C \nl u \nr_{H^r} \nl v \nr_{H^{r'}} .$$
By duality it is sufficient to prove 
$$\left\langle  \left\langle k\right\rangle^s \frac{k}{1+k^2} \widehat{uv},\widehat{w} \right\rangle_{\ell^2} \leq C \nl u \nr_{H^r} \nl v \nr_{H^{r'}} \nl w \nr_{L^2} ,$$
that is 
$$I=\sum_{k \in \Z}  k \left\langle k \right\rangle^{s-2}  \widehat{uv}(k)  \overline{\widehat{w}}(k) \leq C \nl u \nr_{H^r} \nl v \nr_{H^{r'}} \nl w \nr_{L^2} .$$
Let $f(k)=\left\langle k\right\rangle^r \widehat{u}(k) $, $g(k)= \left\langle k\right\rangle^{r'} \widehat{v}(k)$ and $h(k)=k \left\langle k\right\rangle^{-2(1+r+r'-2s)} \overline{\widehat{w}}(k) $. Since
$$\widehat{uv}(k)=\sum_{l\in\Z} \widehat{u}(l) \widehat{v}(k-l)$$
we have
$$I=\sum_{k \in \Z} \sum_{l \in \Z} \frac{\left\langle k\right\rangle^{-3s+2r+2r'}}{\left\langle l\right\rangle^r \left\langle k-l\right\rangle^{r'}} f(l) g(k-l) h(k) .$$
We have $-2s+r+r' \leq 0$ and $-s+r\leq 0$ and  $-s+r' \leq 0$ so \\
$-3s+2r+2r' = -2s+r+r' +(- s + r') +r \leq r$ and $-3s+2r+2r' \leq r'$. \\
Hence $\dfrac{\left\langle k\right\rangle^{-3s+2r+2r'}}{\left\langle l\right\rangle^r \left\langle k-l\right\rangle^{r'}}$ is bounded for $k$ and $l$ in $\Z$.
Then (by Cauchy-Schwarz inequality and Young's inequality)
\begin{align*}
I & \lesssim \sum_{k \in \Z} \sum_{l \in \Z} f(l) g(k-l) h(k) \\
  & \lesssim \nl f \nr_{\ell^2} \nl g * h(- \cdot) \nr_{\ell^2} \\
  & \lesssim \nl f \nr_{\ell^2} \nl g \nr_{\ell^2} \nl h \nr_{\ell^1} \\
  & \lesssim \nl u \nr_{H^r} \nl v \nr_{H^{r'}} \nl w \nr_{L^2} \nl \frac{k}{(1+k^2)^{1+r+r'-2s}} \nr_{\ell_k^2} .
\end{align*}

Since $2s-r-r'<1/4$ we have $1+r+r'-2s>3/4$. Hence
 $$\nl \frac{k}{(1+k^2)^{1+r+r'-2s}} \nr_{\ell_k^2} <+\infty .$$
\end{proof}

In subsection \ref{subsec:H1} we will use this lemma in the particular case $r=r'=s\geq0$, that is 
$$\nl \vphi(D)(uv) \nr_{H^s} \leq C_{s} \nl u \nr_{H^s} \nl v \nr_{H^s} $$
whereas in subsection \ref{subsec:H2} and \ref{subsec:H3} we will need the general case $0\leq r,r' < s$.

\begin{lem}
\label{lem:bilin2}
Let $u \in H^r(\T)$ and $v\in H^s(\T)$ with $0\leq s \leq r$ and $r>\frac12$, then 
$$\nl \vphi(D)(uv) \nr_{H^{s+1}} \leq C \nl u \nr_{H^r} \nl v \nr_{H^s} .$$
\end{lem}

\begin{proof}
Since $r>\frac12$ and $r\geq s \geq 0$, the elements of $H^r(\T)$ are multipliers in $H^s(\T)$, which is to say 
$$\nl uv \nr_{H^s} \lesssim \nl u \nr_{H^r} \nl v \nr_{H^s} .$$
Hence
\begin{align*}
\nl \vphi(D)(uv) \nr_{H^{s+1}} & = \nl \frac{\left\langle k \right\rangle^{s+1} k}{\left\langle k \right\rangle^2} \widehat{uv} \nr_{\ell_k^2} \\
& \leq \nl \left\langle k \right\rangle^s \widehat{uv} \nr_{\ell_k^2} \\
&= \nl uv \nr_{H^{s}} \\
&\lesssim \nl u \nr_{H^r} \nl v \nr_{H^s} .
\end{align*}
\end{proof}

\subsection{Hamiltonian formalism for BBM equation}

Recall that BBM equation reads 
$$u_t+u_x+uu_x-u_{txx}=0 .$$
Let us prove that BBM equation is a hamiltonian equation (\ref{eqn:EDP}). 

First BBM can be written
$$u_t=-\partial_x (1-\partial_x^2)^{-1}(u+\frac{u^2}{2}) .$$

Denote $Z=H^{1/2}_0(\T)=\left\lbrace u \in H^{1/2} / \int_{\T}  u = 0 \right\rbrace $ with the following norm 
$$\nl u \nr_Z = \sum_{k \in \Z \backslash \{ 0 \} } \frac{1+k^2}{k}(a_k^2 + b_k^2) $$
where $a_k$ and $b_k$ are the (real) Fourier coefficients of $u$.

Consider the Hilbert basis of Z given by 
$$ \vphi_n^{+}(x)=\sqrt{\frac{n}{\pi(n^2+1)}} \cos(nx), \qquad \vphi_n^{-}(x)=\sqrt{\frac{n}{\pi(n^2+1)}} \sin(nx) .$$
We have $Z_+=H^{1/2+\eps}_0 < H^{1/2}_0 < H^{1/2-\eps}_0=Z_-$, where $\eps>0$ will be fixed later.

Define 
$$H(u)=\int_{\T} \! \left( \frac{u(x)^2}{2} + \frac{u(x)^3}{6} \right) dx ,$$
we have 
$$\nabla_{\!\! L^2} H(u)= u + \frac{u^2}{2} .$$

Assume 
$$u(t)= \sum_{n} p_n(t) \vphi_n^++q_n(t) \vphi_n^-$$
and 
$$\nabla_{\!\! L^2} H(u)=  \sum_{n} \alpha_n \vphi_n^++\beta_n \vphi_n^- .$$
Denoting $\widetilde{H}(p,q)=H(\sum_{n} p_n(t) \vphi_n^++q_n(t) \vphi_n^-)$ we deduce that
$$\frac{\partial \widetilde{H} }{\partial p_n}= \left\langle \nabla_{\!\! L^2} H(u), \vphi_n^+ \right\rangle_{L^2} = \alpha_n \nl \vphi_n^+ \nr_{L^2}^2= \frac{n \alpha_n}{1+n^2}$$ 
and 
$$\frac{\partial \widetilde{H} }{\partial q_n}=\frac{n \beta_n}{1+n^2} .$$ 

Hence
\begin{align*}
\dot{u}=\sum_{n} \dot{p}_n \vphi_n^++\dot{q}_n \vphi_n^-&= (1-\partial_x^2)^{-1} \partial_x (- \nabla_{\!\! L^2} H(u)) \\
& = \sum_{n} \frac{-n \alpha_n}{1+n^2} \vphi_n^- + \frac{n \beta_n}{1+n^2} \vphi_n^+ 
\end{align*}
so
$$\left\{ \begin{array}{l} \dot{p}_n = \dfrac{n \beta_n}{1+n^2} = \dfrac{\partial \widetilde{H} }{\partial q_n} \\
            \dot{q}_n = \dfrac{- n \alpha_n}{1+n^2} = - \dfrac{\partial \widetilde{H} }{\partial p_n} 
          \end{array} \right.$$

That is $\dot{u}=J \nabla_{\!\! Z} H(u)$.

\subsection{Verification of (H1)}
\label{subsec:H1}

\subsubsection{Local well-posedness}

Recall that $\vphi(k)=\frac{k}{1+k^2}$, the equation (\ref{eq:bbm}) can be written in the form :
\begin{equation}
\left\{ \begin{array}{l}
iu_t=\vphi(D)u+\frac12 \vphi(D)u^2 \\
u(0,x)=u_0(x)
\end{array}
\right.
\label{eq:bbm2}
\end{equation}

Let $e^{-it\vphi(D)}$ be the unitary group defining the associated free evolution. That is, $e^{-it\vphi(D)}u_0$ solves the Cauchy problem 
\begin{equation}
\left\{ \begin{array}{l}
iu_t=\vphi(D)u \\
u(0,x)=u_0(x)
\end{array}
\right.
\label{eq:bbmlin}
\end{equation}
Then, (\ref{eq:bbm2}) may be rewritten as the integral equation 
$$u(t)=e^{-it\vphi(D)}u_0-\frac{i}2 \int_0^t \!\! e^{-i(t-\tau)\vphi(D)}\vphi(D)(u(\tau)^2) d\tau=\mathcal{A}(u)(t,\cdot) .$$

Let $X^s_T=C^0([-T,T],H^s(\T))$. The $H^s$ norm is clearly preserved by the free evolution, thus
\begin{equation}
\nl e^{-it\vphi(D)}u_0 \nr_{X^s_T} = \nl u \nr_{H^s} .
\label{eq:Hsconservation}
\end{equation}

\begin{thm}
\label{thm:bbmexistloc}
Let $s\geq0$. For any $u_0 \in H^s(\T)$, there exist a time $T$ (depending on $u_0$) and a unique solution $u \in X^s_T$ of (\ref{eq:bbm}). The maximal existence time $T_s$ has the property that
$$T_s \geq \frac1{4 C_s \nl u_0 \nr_{H^s}}$$
with $C_s$ the constant from lemma \ref{lem:bilin1} (in the special case $r=r'=s$).

Moreover, for $R>0$, let $T$ denote a uniform existence time for (\ref{eq:bbm}) with $u_0\in B_R(H^s(\T))$, then the map $\Phi:u_0 \mapsto u$ is real-analytic from $B_R(H^s(\T))$ to $X_T^s$.
\end{thm}

\begin{proof}
Let $R=2\nl u_0 \nr_{H^s}$. For any $u \in B_R(X^s_T)$, by (\ref{eq:Hsconservation}) and lemma \ref{lem:bilin1} (with $r=r'=s$) we have
\begin{align*}
\nl \mathcal{A}(u) \nr_{X^s_T} & \leq \nl e^{-it\vphi(D)}u_0 \nr_{X^s_T} +\frac12 \nl \int_0^t \!\! e^{-i(t-\tau)\vphi(D)}\vphi(u(\tau)^2) d\tau \nr_{X^s_T} \\
& \leq \nl u_0\nr_{H^s} + \frac{C_s T}{2} \nl u \nr_{X^s_T}^2 \\
& \leq \nl u_0\nr_{H^s} + \frac{C_s T}{2} R^2 \\
& \leq R \quad \quad \quad \text{for $T = \frac2{C_sR}$}
\end{align*}
and for any $u,v \in B_R(X^s_T)$, by lemma \ref{lem:bilin1} (with $r=r'=s$) we have
\begin{align*}
\nl \mathcal{A}(u)-\mathcal{A}(v) \nr_{X^s_T} & \leq \frac{C_sT}{2} \nl u-v \nr_{X^s_T} \nl u+v \nr_{X^s_T} \leq C_sTR \nl u-v \nr_{X^s_T} . 
\end{align*}
Hence, $\mathcal{A}$ is a contraction mapping of $B_R(X^s_T)$ for $T=\frac{1}{2C_s R}=\frac{1}{4C_s \nl u_0 \nr_{H^s}}$. Thus $\mathcal{A}$ has a unique fixed point which is a solution of (\ref{eq:bbm}) on time interval $[-T,T]$.

Let us consider now the smoothness of $\Phi$. Let $\Lambda:H^s(\T) \times X^s_T \longrightarrow X^s_T$ be defined as
$$\Lambda(u_0,v)(t)=v(t)-e^{-it\vphi(D)}u_0-\frac{i}2 \int_0^t \!\! e^{-i(t-\tau)\vphi(D)}\vphi(D)(v(\tau)^2) d\tau .$$
Due to lemme \ref{lem:bilin1} (with $r=r'=s$), $\Lambda$ is a smooth map from $H^s(\T) \times X^s_T$ to $ X^s_T$. Let $u\in X_T^s$ be the solution of (\ref{eq:bbm}) with initial data $u_0\in H^s(\T)$, which is to say $\Lambda(u_0,u)=0$. Thus, the Fr\'echet derivative of $\Lambda$ with respect to the second variable is the linear map :
$$\Lambda'(u_0,u)(t)[h] = h- \int_0^t \!\! e^{-i(t-\tau)\vphi(D)}\vphi(D)(u(\tau)h(\tau)) d\tau .$$
Still by lemma \ref{lem:bilin1} we get 
$$\nl \int_0^t \!\! e^{-i(t-\tau)\vphi(D)}\vphi(D)(u(\tau)h(\tau)) d\tau \nr_{X^s_T} \leq C T \nl u \nr_{H^s} \nl h \nr_{H^s} .$$
So, for $T'$ sufficiently small (depending only on $\nl u \nr_{H^s}$), $\Lambda'(u_0,u)(t)$ is invertible since it is of the form $Id + K$ with 
$$\nl K \nr_{\mathcal{B}(X^s_{T'},X^s_{T'})} <1$$
where $\mathcal{B}(X^s_{T'},X^s_{T'})$ is the Banach space of bounded linear operators on $X^s_{T'}$. Thus $\Phi:B_R(H^s(\T)) \rightarrow X_{T}^s$ is real-analytic by Implicit Function Theorem.
\end{proof}

\subsubsection{Global well-posedness}

\begin{thm}
\label{thm:bbmexistglob}
The solution defined in theorem \ref{thm:bbmexistloc} is global in time.
\end{thm}

\begin{proof}
Fix $T>0$. The aim is to show that corresponding to any initial data $u_0 \in H^s$, there is a unique solution of (\ref{eq:bbm}) that lies in $X^s_T$. Because of theorem \ref{thm:bbmexistloc}, this result is clear for data that is small enough in $H^s$, and it is sufficient to prove the existence of a solution corresponding to initial data of arbitrary size (uniqueness is a local issue). Fix $u_0 \in H^s$ and let $N$ be such that 
$$\sum_{|k|\geq N} \!\! \left\langle k\right\rangle^{2s} \left|\widehat{u_0}(k)\right|^2  \leq T^{-2} .$$
Such values of $N$ exist since $\left\langle k\right\rangle^{s} \left|\widehat{u_0}(k)\right|$ is in $\ell^2$. Define
$$v_0(x)=\sum_{|k|\geq N} e^{ixk} \widehat{u_0}(k) .$$

By theorem \ref{thm:bbmexistloc}, there exists a unique $v\in X^s_T$ solution of (\ref{eq:bbm}) with initial data $v_0$. Split the initial data $u_0$ into two pieces: $u_0=v_0+w_0$; and consider the following Cauchy problem (where $v$ is now fixed) 
\begin{equation}
\left\{
\begin{array}{l}
w_t-w_{xxt}+w_x+ww_x+(vw)_x \\
w(0,x)=w_0(x)
\end{array}
\right.
\label{eq:bbmw}
\end{equation}
If there exists a solution $w$ of (\ref{eq:bbmw}) in $X^s_T$ then $v+w$ will be a solution of (\ref{eq:bbm}) in $X^s_T$.

First, $w_0$ is in $H^r(\T)$ for all $r>0$, in particular $w_0 \in H^1(\T)$. And (\ref{eq:bbmw}) may be rewritten as the integral equation 
$$w(t,x)=e^{-it\vphi(D)}w_0-\frac{i}2 \int_0^t \!\! e^{-i(t-\tau)\vphi(D)}\vphi(D)(vw+w^2) d\tau=\mathcal{K}(w) .$$
This problem can be solved locally in time on $H^1(\T)$ by the same arguments used to prove theorem \ref{thm:bbmexistloc}. Indeed for any $w\in B_R(X_S^1)$, by lemma \ref{lem:bilin2} (with $r=1$ and $s=0$) and lemma \ref{lem:bilin1} (with $r=r'=s=1$)
\begin{align}
\label{eqn:K1}
\nonumber \nl \mathcal{K}(w) \nr_{X_S^1} & \leq \nl w_0 \nr_{H^1} + C S \left(\nl v \nr_{X_S^0} \nl w \nr_{X_S^1} + \nl w \nr_{X_S^1}^2\right)\\
& \leq C S \nl v \nr_{X_S^0} R
\end{align}
and for any $w_1$ and $w_2$ in $B_R(X_S^1)$
\begin{align}
\label{eqn:K2}
\nonumber \nl \mathcal{K}(w_1) \right.   &-  \left. \mathcal{K}(w_2) \nr_{X_S^1} \\ 
\nonumber &\leq C S \left(\nl v \nr_{X_S^0} \nl w_1-w_2 \nr_{X_S^1} + \nl w_1-w_2 \nr_{X_S^1} \nl w_1+w_2 \nr_{X_S^1} \right)  \\
&\leq C S \left(\nl v \nr_{X_S^0} +2R\right) \nl w_1-w_2 \nr_{X_S^1} .
\end{align}
Hence, by (\ref{eqn:K1}) and (\ref{eqn:K2}), $\mathcal{K}$ has a unique fixed point in $X_S^1$. Therefore we have a solution $w$ in $X^1_S$ for a small time $S$.

If we have an \emph{a priori} bound on the $H^1$-norm of $w$ showing it was bounded on the interval $[-T,T]$ it would follow that a solution on $[-T,T]$ could be obtained.

The formal steps of this inequality are as folllows (the justification is made by regularizing). Multiply the equation (\ref{eq:bbmw}) by $w$, integrate over $\T$, and after integration by parts we get
$$\frac12 \frac{d}{dt} \int_{\T} \!\! \left( w(t,x)^2 + w_x(t,x)^2 \right) dx - \int_{\T} \!\! v(t,x)w(t,x)w_x(t,x) dx=0 .$$
By H\"older and Sobolev inequalities we deduce
\begin{align*}
 \left| \int_{\T} \!\! v(t,x)w(t,x)w_x(t,x) dx \right| & \leq \nl v(t,\cdot) \nr_{L^2} \nl w(t,\cdot) \nr_{L^{\infty}} \nl w_x(t,\cdot) \nr_{L^2} \\
& \leq C \nl v(t,\cdot) \nr_{L^2} \nl w(t,\cdot) \nr_{H^1}^2 .
\end{align*}
Hence
$$\frac{d}{dt} \nl w(t,\cdot) \nr_{H^1}^2 \leq  2C  \nl v(t,\cdot) \nr_{L^2} \nl w(t,\cdot) \nr_{H^1}^2$$
and by Gronwall's inequality 
$$\nl w(t,\cdot) \nr_{H^1} \leq  \nl w_0 \nr_{H^1} \exp \left( C \int_0^t \!\! \nl v(\tau,\cdot) \nr_{L^2} d \tau  \right) .$$

We deduce from this \emph{a priori} bound that the solution $w$ of (\ref{eq:bbmw}) exists on the interval $[-T,T]$, and $v+w$ is a solution of (\ref{eq:bbm}) in $X^s_T$.
\end{proof}

\subsection{Verification of (H2)}
\label{subsec:H2}
\begin{prop}
\label{prop:H2}
For any $T>0$, $R>0$, and $s>0$ there exists $R'$ such that 
$$ \forall 0 \leq t \leq T, \Phi_t(B_R(H^{s})) \subset B_{R'}(H^{s}) .$$
\end{prop}

With $s=\frac12$ we deduce that $\Phi$ satisfies (H2).

\begin{proof}
The result is clear for $s\geq1$, so we assume that $0<s<1$. Fix $T>0$, $R>0$ and $u_0$ in $H^{s}$ such that $\nl u_0 \nr_{H^{s}} \leq R$. Using the same idea as in theorem \ref{thm:bbmexistglob} split $u_0$ into two pieces $u_0=v_0+w_0$, where
$$v_0=\sum_{|k|\geq N} \widehat{u_0}(k) e^{ikx} .$$
Using the same notations, let $v$ be the solution of BBM equation with the initial data $v_0$ and $w$ the solution of (\ref{eq:bbmw}). We want to control $v$ and $w$ in $H^{s}$-norm. 

Fix $\eps>0$ such that $\eps<1/8$ and $s-\eps>0$, we have
$$\nl v_0 \nr_{H^{s-\eps}} \leq N^{-\eps} \nl v_0 \nr_{H^{s}} .$$
We choose $N={\left(\frac{4RC}{T} \right)}^{1/\eps}$ where $C$ is the constant of lemma \ref{lem:bilin1}. Hence we have 
$$ \nl v_0 \nr_{H^{s-\eps}} \leq \frac1{4CT} = M.$$
By local theory (theorem \ref{thm:bbmexistloc}) the flow map 
$$\Phi : B_M(H^{s-\eps}) \longrightarrow X^{s-\eps}_T$$ 
is continuous. Since $H^{s} \cap B_M(H^{s-\eps})$ is precompact in $B_M(H^{s-\eps})$ we have 
$$\sup_{v_0 \in H^{s} \cap B_M(H^{s-\eps})} \nl \Phi(v_0) \nr_{X^{s-\eps}} = C_1(R,T) .$$
By lemma \ref{lem:bilin1} with $r=r'=s-\eps$ we have
$$\nl v \nr_{X^s} \leq \nl v_0 \nr_{H^s} + CT \nl v \nr^2_{X^{s-\eps}} \leq R + CT C_1(R,T)^2 = C_2(R,T).$$

The \emph{a priori} bound on $w$ gives
\begin{align*}
 \nl w(t) \nr_{H^s}\leq  \nl w(t) \nr_{H^1} & \leq \nl w_0 \nr_{H^{1}} \exp \left( C \int_0^t \!\! \nl v(\tau,\cdot) \nr_{L^2} d\tau \right)\\
& \leq N^{1-s} \nl w_0 \nr_{H^{s}} e^{CTC_2(R,T)}\\
& \leq C_3(R,T) .
\end{align*}
Hence, we have
$$ \nl u \nr_{X_T^{s}}  \leq C_2(R,T)+C_3(R,T)$$
\end{proof}

\begin{cor}
For each $T>0$ and $s>0$, the flow map $\Phi:H^{s} \rightarrow X^{s}_T$ is real analytic. 
\end{cor}

\begin{proof}
Let $u_0 \in H^{s}$, $R=\nl u_0 \nr_{H^{s}}$ and $T>0$. By proposition \ref{prop:H2}, there exists $R'$ such that $\Phi_t(B_{2R}(H^{s})) \subset B_{R'}(H^{s})$, for all $t\in[0,T]$. And by local theory (theorem \ref{thm:bbmexistloc}) there exists a small time $\tau$ such that $\Phi:B_{R'}(H^{s}) \rightarrow X^{s}_{\tau}$ is real analytic. Splitting the time intervalle $[0,T]$ into $\bigcup [k \tau ,(k+1)\tau]$, we deduce that $\Phi:H^{s} \rightarrow X^{s}_T$ is real analytic.
\end{proof}

\subsection{Verification of (H3)}
\label{subsec:H3}

Recalll that $\widetilde{\Phi}$ denote the non-linear part of the flow, that is $\Phi_t=e^{-it\vphi(D)}(I+\widetilde{\Phi}_t)$. The assumption (H3) results from 
\begin{prop}
\label{prop:compactness}
For any $u_0,v_0 \in B_R(H^{1/2}(\T))$ we have the following estimate 
$$\nl \widetilde{\Phi}(u_0)-\widetilde{\Phi}(v_0) \nr_{X_T^{1/2+\eps}} \leq C_{R,T,\eps} \nl u_0-v_0\nr_{H^{1/2-\eps}}$$
for $0<\eps<1/12$.
\end{prop}

\begin{proof}
Let $0<\eps<\frac1{12}$, $u_0$ and $v_0$ in $B_R(H^{1/2})$. Denoting $u$ and $v$ the solutions of BBM equation with initial data $u_0$ and $v_0$. By lemma \ref{lem:bilin1} with $s=\frac12+\eps$ and $r=\frac12$ and $r'=\frac12-\eps$ and (H2) we have
\begin{align*}
 \nl \widetilde{\Phi}_t(u_0)-\widetilde{\Phi}_t(v_0) \nr_{X_T^{1/2+\eps}} & \leq CT \nl u+v \nr_{X_T^{1/2}} \nl u-v \nr_{X_T^{1/2-\eps}} \\
& \leq 2CT R'_{R,T}  \nl u-v \nr_{X_T^{1/2-\eps}} .
\end{align*}
Since $u_0$ and $v_0$ are in $B_R(H^{1/2})$ and $\Phi$ is $C^1$ on $B_R(H^{1/2})$ which is a relatively compact subset of $H^{1/2-\eps}$ we have
\begin{align*}
\nl u-v \nr_{X_T^{1/2-\eps}} &= \nl \Phi_t(u_0)-\Phi_t(v_0) \nr_{X_T^{1/2-\eps}} \\
& \!\!\!\!\!\!\!\!\!\!\!\!  \leq \sup_{ w_0 \in B_R(H^{1/2}) \cap H^{1/2-\eps}} \! \left( \! \nl d \Phi (w_0) \nr_{\mathcal{B}\left(H^{1/2-\eps},X^{1/2-\eps}_T\right)} \! \right) \! \nl u_0-v_0 \nr_{H^{1/2-\eps}} \\
& \!\!\!\!\!\!\!\!\!\!\!\! \leq  C_{R,T,\eps} \nl u_0-v_0 \nr_{H^{1/2-\eps}} .
\end{align*}

\end{proof}

Hence, we can apply the non-squeezing theorem (theorem \ref{thm:nonsqueezing}) and that proves the theorem \ref{thm:BBMnonsqueezing}.

\hspace{5mm}

\textbf{Acknowledgments:} I'm grateful to Nikolay Tzvetkov for introducing me to this subject and for his advices on my work. I would also like to thank Patrick G\'erard for many helpful discussions.

I thank the referee for pointing out an error in a previous version of this paper.

\bibliographystyle{amsplain}
\bibliography{biblio}

\providecommand{\bysame}{\leavevmode\hbox to3em{\hrulefill}\thinspace}
\providecommand{\MR}{\relax\ifhmode\unskip\space\fi MR }
\providecommand{\MRhref}[2]{%
  \href{http://www.ams.org/mathscinet-getitem?mr=#1}{#2}
}
\providecommand{\href}[2]{#2}
\begin{thebibliography}{1}

\bibitem{BBM}
Thomas~B. Benjamin, Jerry~L. Bona, and John~J. Mahony, \emph{Model equations
  for long waves in nonlinear dispersive systems}, Philosophical Transactions
  of the Royal Society of London \textbf{272} (1972), no.~1220, 47--78.

\bibitem{BT}
Jerry~L. Bona and Nikolay Tzvetkov, \emph{Sharp well-posedness results for the
  {BBM} equation}, Discrete and Continuous Dynamical Systems \textbf{23}
  (2009), no.~4, 1241--1252.

\bibitem{bourgain}
Jean Bourgain, \emph{Aspects of long time behaviour of solutions of nonlinear
  {H}amiltonian evolution equations}, Geometric and Functional Analysis
  \textbf{5} (1995), no.~2, 105--140.

\bibitem{CKSTT}
James Colliander, Markus Keel, Gigliola Staffilani, Hideo Takaoka, and Terence
  Tao, \emph{Symplectic nonsqueezing of the {K}orteweg-de {V}ries flow}, Acta
  Mathematica \textbf{195} (2005), no.~2, 197--252.

\bibitem{gromov}
Mikhail Gromov, \emph{Pseudo-holomorphic curves in symplectic manifolds},
  Inventiones Mathematicae \textbf{82} (1985), no.~2, 307--347.

\bibitem{HZ}
Helmut Hofer and Eduard Zehnder, \emph{Symplectic {I}nvariants and
  {H}amiltonian {D}ynamics}, Birkh\"auser, 1994.

\bibitem{kuksin}
Serge\"i Kuksin, \emph{Infinite-dimensional symplectic capacities and a
  squeezing theorem for {H}amiltonian {PDE}'s}, Communications in Mathematical
  Physics \textbf{167} (1995), 531--552.

\bibitem{panthee}
Mahendra Panthee, \emph{On the ill-posedness result for the {BBM} equation},
  arXiv:1003.6098v1, preprint 2010.

\end{thebibliography}

\end{document}